\documentclass{amsart}
\usepackage{amssymb}
\usepackage{amsmath,amscd}
\usepackage{graphicx}
\usepackage{color,hyperref}
\usepackage{tikz}
\usetikzlibrary{matrix,arrows}

\newcommand{\mf}{\mathfrak}
\newcommand{\mc}{\mathcal}
\newcommand{\mb}{\mathbf}

\newcommand{\R}{\mathbf R}
\newcommand{\C}{\mathbf C}
\newcommand{\Q}{\mathbf Q}
\newcommand{\Z}{\mathbf Z}
\newcommand{\F}{\mathbf F}
\newcommand{\N}{\mathbf N}
\newcommand{\K}{\mathbf K}

\newcommand{\Fix}{\textnormal{Fix}}

\newcommand{\Frob}{\textnormal{Frob}}

\numberwithin{equation}{section}

\theoremstyle{plain}
\newtheorem{theorem}{Theorem}
\newtheorem{lemma}[theorem]{Lemma}

\newtheorem{corollary}[theorem]{Corollary}

\theoremstyle{definition}

\begin{document}

\title[Exceptional Polynomial Mappings]{On the Arithmetic Exceptionality of 
Polynomial Mappings}

\author{\"{O}mer K\"{u}\c{c}\"{u}ksakall{\i}}
\address{Middle East Technical University, Mathematics Department, 06800 Ankara,
Turkey.}
\email{komer@metu.edu.tr}

\date{\today}

\begin{abstract}
In this note we prove that certain polynomial mappings $P_\mf{g}^k(\mb{x}) \in 
\Z[\mb{x}]$ in $n$-variables obtained from simple complex Lie algebras $\mf{g}$ 
of arbitrary rank $n \ge 1$,  are exceptional. 
\end{abstract}



\maketitle


We recall that a polynomial mapping $P \in {\Z}[\mb{x}]$ in $n$ variables is 
said to be exceptional if the reduced map $\bar{P} : {\F}_p^n \to {\F}_p^n$ is 
a permutation for infinitely many primes $p$. Lidl and Wells proved the 
existence of nontrivial exceptional polynomial mappings of arbitrary rank 
\cite{lidlwells}. They achieved this by means of elementary methods, namely 
using the theory of symmetric functions, however their construction can be 
related to the simple complex Lie algebras $A_n$ \cite{hoffwith}.

In this paper, we prove that certain polynomial mappings $P_\mf{g}^k(\mb{x}) 
\in \Z[\mb{x}]$ in $n$-variables obtained from arbitrary simple complex Lie 
algebras $\mf{g}$ of rank $n \ge 1$ are exceptional. The structure 
of the proof follows closely the pattern of the $n = 2$ case 
\cite{kucuksakalli-bivariate}.

Let $\mf{g}$ be a simple complex Lie algebra of rank $n$ and $\mf{h}$ its Cartan 
subalgebra, $\mf{h}^*$ its dual space, $\mc{L}$ a lattice of weights in 
$\mf{h}^*$ generated by the fundamental weights $\omega_1, \ldots, \omega_n$, 
and $L$ the dual lattice in $\mf{h}$. Applying the exponential form of 
Chevalley's Theorem (Thm. 1, p.188, \cite{Bourbaki}) one proves that the 
quotient of $\mf{h}/L$ under the action of the Weyl group $W$ (induced from the 
action of $W$ on $\mc{L}$) is the $n$-dimensional complex affine space and the 
quotient 
map is given by 
$\varPhi_\mf{g}:\mf{h}/L \rightarrow \C^n$, $\varPhi_\mf{g} = 
(\varphi_1,\ldots, \varphi_n)$
\[\varphi_k(\mb{x})=\sum_{w\in W} e^{2\pi i w (\omega_k)(\mb{x})}.\] 
This  construction leads to the following result first given by Veselov, and 
somewhat later by Hofmann and Withers independently.

\begin{theorem}[\cite{veselov},\cite{hoffwith}]\label{veselov}
With each simple complex Lie algebra $\mf{g}$ of rank $n$, there is associated 
an infinite sequence of integrable polynomial mappings $P_\mf{g}^k, ~k \in {\N}$ 
 determined from the conditions
\[ \varPhi_\mf{g}(k\mathbf{x})=P_\mf{g}^k(\varPhi_\mf{g}(\mathbf{x})). \]
All coefficients of the polynomials defining $P_\mf{g}^k$ are integers.
\end{theorem}

We will prove that for any $\mf{g}$, there exists $k$ such that the mapping 
$P_\mf{g}^k$ is exceptional (Corollary~\ref{exist}). \\

We first fix our notation.

Throughout the paper $q$ denotes a power of a prime $p$. 

$\Frob_q$ is the Frobenius map  $(x_1,\ldots,x_n) \mapsto (x_1^q, \ldots, 
x_n^q)$. 

For any polynomial $P\in \Z[\mb{x}]$ in the $n$ variables $x_1,\ldots, x_n$,
$\bar{P}:\F_q^n \rightarrow \F_q^n$ is the map induced by reduction mod $p$. 

$\K = \Q( \Fix(P_\mf{g}^q) )$ is the number field which is obtained by adjoining 
the coordinates of the fixed points of $P_\mf{g}^q$ to the rational numbers. 

$\mf{p}$ is a prime ideal of $\K$ lying over $p$. 

For $\mf{g}$ a simple complex Lie algebra of rank $n$ with roots $\lambda_i, ~ 
i= 1,\ldots,n$ we identify $\mf{h} = \oplus \C \lambda_i$ (resp. the lattice $L 
= \oplus \Z\lambda_i$) with $\C^n$ (resp. $\Z^n$).

$I_n$ is the identity matrix of dimensions $n\times n$. For $w \in W$, $T_w$ is 
the $n\times n$ matrix representing the endomorphism $T_w : \mc{L} \to \mc{L}$ 
defined by  $T_w(\omega_i) = 
w(\omega_i)$ for each 
$i=1, \ldots, n$.\\

The following commutative diagram summarizes the set-up we will work in. 
\begin{center}
\begin{tikzpicture}
  \matrix (m) [matrix of math nodes,row sep=1.5em,column sep=10em,minimum 
width=2em] {
     \mf{h}/L & \C^n \\
     \mf{h}/L & \C^n \\};
  \path[-stealth]
    (m-1-1) edge node [left] {[k]} (m-2-1)
     edge node [above] {$\varPhi_\mf{g}$} (m-1-2)
    (m-2-1) edge node [above] {$\varPhi_\mf{g}$} (m-2-2)
    (m-1-2) edge node [right] {$P_\mf{g}^k$} (m-2-2);
\end{tikzpicture}
\end{center}

With this notation, our main result is the following theorem.

\begin{theorem}\label{main}
Let $\mf{g}$ be a complex Lie algebra of rank $n$ and let $W$ be its Weyl 
group. The polynomial mapping $\bar{P}_\mf{g}^k:\F_q^n \rightarrow \F_q^n$ is a 
permutation if and only if $qI_n-T_w$ is invertible modulo $k$ for each $w\in 
W$.  
\end{theorem}

The proof of the theorem will be given at the end of the paper. We note the 
following corollaries.

\begin{corollary}\label{order}
Let $\mf{g}$ be a complex Lie algebra of rank $n$ and let $W$ be its Weyl 
group. The polynomial mapping $\bar{P}_\mf{g}^k:\F_q^n \rightarrow \F_q^n$ is a 
permutation if $\gcd(k,q^s-1)=1$ for each $s\in \{ \textnormal{order of }w 
\mathrel{|}w\in W \}.$
\end{corollary}
\begin{proof}
The matrix $T_w$ satisfies $T_w^s=I_n$ where $s$ is the order of $w\in W$. The 
matrix $(q^s-1)I_n$ is invertible modulo $k$ by the hypothesis. Note that
\[ (q^s-1)I_n = (qI_n)^s-T_w^s = (qI_n-T_w)(q^{s-1}I_n+\ldots+T_w^{s-1}). \]
Thus $qI_n-T_w$ is invertible modulo $k$. 
\end{proof}

In Corollary~\ref{order} the converse implication is not true. Lidl and Wells prove in 
\cite{lidlwells} that $P_{A_4}^k$ is a permutation of 
$\F_q^4$ if and only if $\gcd(k,q^s-1)=1$ for $s=1,2,3,4,5$. On the other hand, 
the Weyl group $A_4$ is the symmetric group $S_5$ which contains 
an element of order $6$. \\

An important application of Theorem~\ref{main} is the following corollary.

\begin{corollary}\label{exist}
There exists $k \in \N$ such that $P_\mf{g}^k$ is exceptional.
\end{corollary}
\begin{proof}
There exist infinitely many primes in any arithmetic progression. It is now easy 
to see by Corollary~\ref{order} that for each $\mf{g}$, there exists an integer 
$k$ so that $P_\mf{g}^k$ is exceptional. 
\end{proof}

In the proof of the theorem we will need the following lemmata.

\begin{lemma}\label{fix}
For any integer $k \geq 1$, we have
$|\Fix(P_\mf{g}^k)|=k^n$.
\end{lemma}
\begin{proof} 
In order to prove this lemma, we generalize an idea of Uchimura for the case 
$\mf{g}=A_2$ \cite{uchimura}. For an illustration of this idea, see 
\cite{kucuksakalli-bivariate}.

The set of points in $\C^n$ with bounded orbit under $P_\mf{g}^k$ are of the 
form $\varPhi_\mf{g}({\mb x})$ where $\mb{x}$ has real components. The Weyl 
group $W$ acts on the quotient set $D=\R^n/\Z^n$. The elements in 
$\varPhi_\mf{g}(D)$ can be represented by several different expressions 
$\varPhi({\mb x})$, with ${\mb x}=(x_1,\ldots,x_n)$ and $0\leq x_i <1$, thanks 
to the action of the Weyl group $W$. Consider the compact set $D/W$. The 
multiplication map $[k]:D/W \rightarrow D/W$ induces a $k^n$ to $1$ map.

Divide $D/W$ into $k^n$ simplexes $T_1,\ldots,T_{k^n}$ such that each one of 
them is mapped onto $D/W$ under the multiplication by $k$. Consider the 
inverse map from $D/W$ to $T_i$ which is division by $k$ together with a 
suitable linear translation. Being a continuous map there exists at least one 
fixed point of this map. Moreover there is at most one such point in each $T_i$ 
because of the linearity. We must show that these fixed points are distinct. A 
repetition can occur only at the boundaries of the simplexes $T_i$. However the 
multiplication by $k$ maps such a boundary to the boundary of $D/W$. It follows 
that a possible repetition may only be at a corner of a simplex which has a 
part on the boundary of $D/W$. However such a corner is mapped to a corner of 
$D/W$.

It remains to show that a common corner $P$ of two different simplexes $T_i$ 
and $T_j$, which is at the same time a corner of $D/W$, is not fixed 
under $[k]$. Assume otherwise. The compact set $D/W$, and therefore each 
simplex, has $n$ edges connecting to a corner. Thus, we see that 
there must be a common edge of $T_{i}$ and $T_j$ with endpoint $P$, and which 
does not lie on the boundary of $D/W$. Note that the outer $n$ edges on the 
boundary of $D/W$ should be permuted among themselves under $[k]$. On the other 
hand, the extra edge shall be mapped to the boundary of $D/W$ onto one of the 
outer edges. Now, $T_i$ has $n$ edges with endpoint $P$, and at least two of 
them are mapped onto the same edge of $D/W$. This is a contradiction.
\end{proof}

\begin{lemma}\label{nf} Let $k\geq 1$ be an integer. The number field $\Q( 
\Fix(P_\mf{g}^k) )$ is contained in the compositum of the cyclotomic fields 
$\Q(\zeta_{k^s-1})$ where $s\in \{ \textnormal{order of }w \mathrel{|} w\in W 
\}.$
\end{lemma}
\begin{proof}
 Let $\alpha = \varPhi_\mf{g}({\mb x})$ be an element that is fixed under 
$P_\mf{g}^k$. It follows that $k\mb{x}\equiv w\mb{x} \pmod{\Z^n}$ for some 
$w\in W$. If the order of $w$ is $s$, then we have $k^s\mb{x}\equiv \mb{x} 
\pmod{\Z^n}$. It follows that $\mb{x}$ is a vector with rational components 
whose denominators are divisors of $k^s-1$.
\end{proof}

\begin{lemma}\label{red}
 Let $\alpha$ and $\beta$ be elements of $\Fix(P_\mf{g}^q)$. If 
$\alpha\equiv\beta\pmod{\mf{p}} $, then $\alpha=\beta$.
\end{lemma}
\begin{proof}
 There exist $\mb{a}, \mb{b}\in {\mc O}_{\mf{p}}^n/\Z^n$, where $\mc O_\mf{p}$
 is the localization at $\mf{p}$ of the ring of integers $\mc O$ of $\K$,  such 
that $\alpha=\varPhi_\mf{g}({\mb a})$ and $\beta=\varPhi_\mf{g}({\mb b})$. The 
components of $\mb{a}$ and $\mb{b}$ have denominators which are divisors of 
$q^s-1$ by the proof of Lemma~\ref{nf}. Note that $\zeta_{q^s-1}^m \equiv 
\zeta_{q^s-1}^{\tilde{m}} 
\pmod{\mf{p}}$ if and only if $m\equiv \tilde{m} \pmod{q^s-1}$. Moreover, the 
map $\varPhi_\mf{g}$ is a composition of elementary symmetric functions, 
invertible linear maps and maps of the form $z+1/z$ \cite{hoffwith}. If 
$\bar{\alpha}=\bar{\beta}$, then this means that $w_1\mb{a} \equiv w_2\mb{b} 
\pmod{\Z^n}$ for some $w_1,w_2\in W$. Thus $\alpha=\varPhi_\mf{g}({w_1\mb 
a})=\varPhi_\mf{g}({w_2\mb b})=\beta$.
\end{proof}

\begin{lemma}\label{frob}
 We have $\bar{P}_\mf{g}^q = \Frob_q$. 
\end{lemma}
\begin{proof}
Let us consider the map $\bar{\varPhi}_\mf{g}:{\mc O}_{\mf{p}}^n/\Z^n \rightarrow 
\bar{\F}_p^n$ given by $\mb{x}\mapsto\overline{\varPhi_\mf{g}(\mb{x})}$. 
This map is surjective. Letting $t_j = e^{2\pi ix_j}, j = 1,\ldots,n$ we 
see that each component $\varphi_k$ of $\varPhi_\mf{g}(\mb{x})$ is given by a 
sum of integer powers of $t_j$'s. It follows that $\varphi_k(q\mathbf{x})$ is 
obtained by raising each term in this sum to its $q$-th power. We have
\begin{align*}
 \bar{P}_\mf{g}^q\left(\overline{\varPhi_\mf{g}(\bf{x})}\right) & = \left( 
\overline{\varphi_1(q\mb{x})},\ldots,\overline{\varphi_n(q\mb{x})} \right) \\
 &= \left( \left(\overline{\varphi_1(\bf{x})}\right)^q,\ldots, 
\left(\overline{\varphi_n(\bf{x})}\right)^q  \right) \\
 &= \Frob_q\left(\overline{\varPhi_\mf{g}(\bf{x})}\right).
\end{align*}
This proves the claim.
\end{proof}

There are $q^n$ fixed points of $P_\mf{g}^q$ by Lemma~\ref{fix}. Each one of 
these elements reduce to a different element in $(\mathcal{O}/\mf{p})^n$ by 
Lemma~\ref{red}. Moreover, each reduced element belongs to $\F_q^n$ by 
Lemma~\ref{frob}. Thus, we have a one-to-one correspondence
\[ \F_q^n \longleftrightarrow \Fix(P_\mf{g}^q) \]
 obtained by reducing the elements in $\Fix(P_\mf{g}^q)$ modulo $\mf{p}$. Note 
that $\mathcal{O}/\mf{p}$ is always a nontrivial extension of $\F_q$. This 
correspondence is compatible under the actions of $\bar{P}_\mf{g}^q$ and 
$P_\mf{g}^q$, respectively.\\

Now, we are ready to prove our main result.

\begin{proof}
Let $\alpha=\varPhi_\mf{g}(\mb{x})$ be an element of $\Fix(P_\mf{g}^q)$. Then, 
we have
\[ P_\mf{g}^q(P_\mf{g}^k(\alpha))= \varPhi_\mf{g}(qk\mb{x}) 
=P_\mf{g}^k(P_\mf{g}^q(\alpha))= 
P_\mf{g}^k(\alpha). \]
Thus, the restricted map $P_\mf{g}^k: \Fix(P_\mf{g}^q) \rightarrow 
\Fix(P_\mf{g}^q)$ is well-defined. The components of $\mb{x}$ 
are rational numbers whose denominators are divisors of $q^s-1$ by the proof 
of Lemma~\ref{nf}. Actually, we can say more about these components. The set 
of fixed points of $P_\mf{g}^q$ is obtained by solving the equation 
$q\mb{x}=w\mb{x} \pmod{\Z^n}$ for each $w\in W$. It is clear that the rows 
$\mb{x}_i^w$ of the matrix $(qI_n-T_w)^{-1}$ generate the set 
$\Fix(P_\mf{g}^q)$. More 
precisely, we have
\[ \Fix(P_\mf{g}^q) = \left\{ \varPhi_\mf{g}\left(\sum_{i=1}^n 
m_i\mb{x}_i^w\right) \mathrel{:} m_i\in\Z, w\in W\right\}. \]

Suppose that $qI_n-T_w$ is invertible modulo $k$ for each $w\in W$. This means 
that the vectors $\mb{x}_i^w$ have rational components whose denominators are 
relatively prime to $k$. Let $d$ be the product of all possible denominators, 
when the components of $\mb{x}_i^w$ are expressed in their lowest terms. Then 
there exists $\ell$ such that $k\ell \equiv 1 \pmod{d} $. As a result 
$P_\mf{g}^k$ 
and $P_\mf{g}^\ell$, restricted to $\Fix(P_\mf{g}^q)$, are inverses of 
each other. Therefore $P_\mf{g}^k$ permutes the finite set $\Fix(P_\mf{g}^q)$.

For the converse, suppose that $P_\mf{g}^k$ permutes the finite set 
$\Fix(P_\mf{g}^q)$. This is possible if the multiplication by $k$ does not kill 
any denominators within the vectors $\mb{x}_i^w$. Therefore, the matrix 
$qI_n-T_w$ must be invertible modulo $k$ for each $w\in W$.
\end{proof}

Veselov believes that the family of maps $P_\mf{g}^k$ exhaust all integrable 
polynomial mappings $\C^n \to \C^n$ of degree $d > 1$ (\cite{veselov}, p.212). 
To the best of our knowledge, no counterexample has been found so far. Relying 
on this conjecture, one expects that the family $P_\mf{g}^k$ together with 
linear mappings exhaust all exceptional mappings in $n$ variables. \\

{\bf Acknowledgement}. It is a pleasure for the author to thank H. \"{O}nsiper for some
useful discussions. \\

{\small
\def\refname{References}
\newcommand{\etalchar}[1]{$^{#1}$}

\end{document}